\documentclass[11pt,]{article}
\usepackage{graphicx}
\usepackage{amssymb}
\usepackage{epstopdf}
\usepackage{amsmath,amsthm,amscd,amssymb}
\usepackage{latexsym}
\usepackage[colorlinks,citecolor=red,pagebackref,hypertexnames=false]{hyperref}
\usepackage{geometry} 
\usepackage{cite}              
\numberwithin{equation}{section}

\theoremstyle{plain}
\newtheorem{theorem}{Theorem}[section]
\newtheorem{lemma}[theorem]{Lemma}

\theoremstyle{definition}

\newtheorem{case[theorem]}{Case}

\theoremstyle{remark}

\numberwithin{equation}{section}

\begin{document}

\title{Sum-Product Type Estimates for Subsets of Finite Valuation Rings}

\author{Esen Aksoy Yazici}

\maketitle

\begin{abstract} Let $R$ be a finite valuation ring of order $q^r.$ Using a point-plane incidence estimate in $R^3$, we obtain sum-product type estimates for subsets of $R$. In particular, we prove that for $A\subset R$, 
$$|AA+A|\gg \min\left\{q^{r}, \frac{|A|^3}{q^{2r-1}}\right\}.$$
We also show that  if $|A+A||A|^{2}>q^{3r-1}$, then
$$|A^2+A^2||A+A|\gg q^{\frac{r}{2}}|A|^{\frac{3}{2}}.$$
 \end{abstract}

\section{Introduction} 
The classical sum-product estimates in additive combinatorics is about comparing the size $|A|$ of finite subset $A$ of a ring $R$ with $|A+A|$ and $|A.A|$ and more generally with $|F(A)|$ 
 for some function $F$ of $A$. It is expected that either $|A+A|$ or $|A.A|$ has large growth compared to $|A|$, unless  $A$ is close to being a subring of $R$. A variant of this problem arises considering the similar questions for combinations of different subsets of $R$. Sum-product variant questions have been of great interest and was considered in many different context by various authors in the literature. For an extensive exploration of the subject, we refer the reader \cite{AMRS, BB, BKT, ES, MNS, NRS, P, RSS, T, TVU} and the references therein.

The growth estimates in finite field setting was recently studied by Aksoy Yazici, Murphy, Rudnev and Shkredov in \cite{AMRS}. The main tool in that paper was a point-plane incidence estimate in $FP^3$ given by Rudnev in \cite{R}, where $F=\mathbb{F}_q$ is a finite field of order $q$, and  $FP^3$ is the $3$-dimensional projective space over $F$. The method was to consider the number of  solution of an energy equation in $F^3$,  which was called as collision, as a point-plane incidence in the same space and then applying the aforementioned incidence bound.

In \cite{TV}, Thang and Vinh used a graph theoretical approach to obtain a point-line incidence bound in $R^2$, where R is a finite valuation ring of order $q^r$. They later used this estimate to improve and generalize the triangle area result given in \cite{AY}  for  $\mathbb{Z}_{q}^2$ to $R$. Here in this paper, we aim to use the graph theoretical approach similar to the Thang and Vinh's argument and obtain a point-plane incidence bound in $R^3$. We then utilize it with an analog set up in \cite{AMRS} to get the following sum-product estimates for subsets of finite valuation rings. 

\vskip.125in

\subsection {Notation.} Throughout $X\ll Y$ means that there exists an absolute constant $c$ such that $X\le cY$, and $"\gg"$ is defined in a similar way. For a ring $R$, we denote the set of units as $R^{*}$ and the set of nonunits as $R^{0}$.

Let $A$ and $B$ be nonempty subsets of $R$. The sum set is defined as 
$$A+B=\{a+b: a\in A, b\in B\} .$$ Similarly, the product set is 
$$AB=\{a.b: a\in A, b\in B\} $$
and 
$$A^n=\{a^n: a\in A\}.$$

\subsection{Statement of Main Results}

In the following two theorems, we modify sum-product results in \cite{AMRS} over finite fields  to prove their analogs over finite valuation rings.

\begin{theorem}\label{sumproduct}
Let $R$ be a finite valuation ring of order $q^r$ and $A, B, C$ be subsets of $R$. Then,
$$|BA+C|\gg \min \left\{q^{r}, \frac{|A||B||C|}{q^{2r-1}}\right\}.$$ In particular,
$$|AA+A|\gg \min\left\{q^{r}, \frac{|A|^3}{q^{2r-1}}\right\},$$ so that if $|A|>q^{r-\frac{1}{3}}$, then $|AA+A|\gg q^{r}$
\end{theorem}

\begin{theorem}\label{second}Let $R$ be a finite valuation ring of order $q^r$. Let $A$ be subset of $R$. If $|A+A||A|^{2}>q^{3r-1}$, then
$$|A^2+A^2||A+A|\gg q^{\frac{r}{2}}|A|^{\frac{3}{2}}.$$
\end{theorem}

\section{Proof of Theorem \ref{sumproduct}}

For the Proof of Theorem \ref{sumproduct} we will need the following two results from spectral graph theory.

\begin{lemma}\emph{\cite[Lemma 2.1]{TV}}\label{graph}
Suppose $G$ is a bipartite graph with parts $A,B$ such that the vertices in $A$
all have degree $a$ and the vertices in $B$ all have degree $b$. For any two sets $X \subset A$ and
$Y \subset B$, the number of edges between $X$ and $Y$, $e(X, Y )$, satisfies
$$\left|e(X,Y)-\frac{a}{|B|} |X||Y|\right|\le \lambda_{3}|X|^{\frac{1}{2}}|Y|^{\frac{1}{2}}$$
where $\lambda_{3}$ is the third eigenvalue of $G$.
\end{lemma}

Erd\H{o}s-Renyi bipartite graphs $E_{q,d}(R)=(A\cup B, E)$ over finite valuation rings are defined as follows: The vertices of $E_{q,d}(R)$ are given by $[x]$, where $x\in R^d\setminus(R^{0})^d$ and $[x]$ is equivalent to $[y]$ if and only if $x=ty$ for some $t\in R^{*}$. There is an edge between $[x]$  and $[y]$ if and only if $x \cdot y=0$

\begin{theorem}\emph{\cite[Theorem 2.4]{TV}} \label{renyi}
 The cardinality of each vertex part of $E_{q,d}(R)$ is $q^{(d-1)(r-1)}$\\$\frac{(q^d-1)}{q-1}$, $deg(A)=deg(B)=q^{(d-2)(r-1)}\frac{(q^{d-1}-1)}{q-1}$, and the third eigenvalue $\lambda_{3}\leq q^{\frac{1}{2}(d-2)(2r-1)}.$
\end{theorem}

In the proof of Theorem \ref{sumproduct} we will use the following incidence theorem between points and planes in $R^3$, which is analogous to the point-line incidence bound in $R^2$ given in \cite[Theorem 4.2]{TV}.

\begin{theorem}\label{incidence}
Let $R$ be a finite valuation ring of order $q^r$.  Let $Q$ be a set of points in $R^3$ and $\Pi$ be  a set of planes in $R^3$. Then the number of incidences  $\left|I(Q, \Pi)\right|$ between $Q$ and $\Pi$ satisfies 
\begin{eqnarray}\label{pq}
\left|\left|I(Q,\Pi)\right|-\frac{1}{q^{r-1}}\frac{(q^{2}+q+1)}{(q^3+q^2+q+1)}|Q||\Pi|\right|\le q^{2r-1}|Q|^{\frac{1}{2}}|\Pi|^{\frac{1}{2}} 
\end{eqnarray}
Hence
$$\left|I(Q,\Pi)\right|\le \frac{1}{q^r}|Q||\Pi|+ q^{2r-1}|Q|^{\frac{1}{2}}|\Pi|^{\frac{1}{2}} .$$
\end{theorem}
\begin{proof}
We can identify the point $(x_1,x_2,x_3)\in Q$ with $[x_1,x_2,x_3,1]\in E_{q,4}(R)$, and the plane $ax+by+cz=d$ in $\Pi$, $(a,b,c,d)\in R^4\setminus (R^{0})^4$, with $[a,b,c,-d]\in E_{q,4}(R)$. Therefore we will identify $Q$ with $Q'$ and $\Pi$ with $ \Pi'$, where $Q'=\{[x_1,x_2,x_3,1]:\;(x_1,x_2,x_3)\in Q\}$ and $\Pi'=\{[a,b,c,-d]:\; ax+by+cz=d\; \text{in}\; \Pi \}$. Note that $|Q|=|Q'|$ and $|\Pi|=|\Pi'|$.

Then it clearly follows that the number of incidences between $Q$ and $\Pi$ is equal to the number of edges between $Q'$ and $\Pi'$ in the Erd\H{o}s-Renyi graph $E_{q,4}(R)$. By Theorem \ref{renyi}, the cardinality of each vertex part of $E_{q,4}(R)$ is $q^{3(r-1)}(q^3+q^2+q+1)$, $deg(A)=deg(B)=q^{2(r-1)}(q^2+q+1)$ and the third eigenvalue $\lambda_3$ of $E_{q,4}$ is at most $q^{2r-1}$. From Lemma \ref{graph}, (\ref{pq}) follows which implies that
\begin{eqnarray}
\left|I(Q,\Pi)\right|\le \frac{1}{q^r}|Q||\Pi|+ q^{2r-1}|Q|^{\frac{1}{2}}|\Pi|^{\frac{1}{2}}
\end{eqnarray}
and completes the proof.
\end{proof}

For the proof of Theorem \ref{sumproduct}, we use a similar set up given in \cite{AMRS} for the sum-product type estimates over finite fields.  For completeness, we recall the notation and give the proof accordingly.

Let $P\subset R^2\setminus \{(0,0)\}.$ Define the set of lines 
$$L=L_P=\{l_{m,b}:(m,b)\in P\}$$
and
$$L(A)=L_P(A)=\{l_{m,b}(a)=ma+b:(m,b)\in P, a\in A\}$$
\begin{eqnarray}\label{energy}
E(L,A)=|\{(l,l',a,a')\in L^2\times A^2: l(a)=l'(a')\}|
\end{eqnarray}
For $L=L_P$, (\ref{energy}) becomes
$$E(L,A)=|\{ma+b=m'a'+b': (m,b), (m',b')\in P, a,a'\in A  \}| $$

We have the following result.

\begin{theorem}\label{collision}
Let $P\subset R^2$ and $A\in R$. Let $L=L_P$. Then 
\begin{eqnarray}
E(L,A)\leq \frac{1}{q^{r}}|L|^2|A|^2+q^{2r-1}|L||A|
\end{eqnarray}
and
\begin{eqnarray}
|L(A)|\gg\min\left\{q^{r}, \frac{|L||A|}{q^{2r-1}}\right\}
\end{eqnarray}
\end{theorem}

\begin{proof}
Define
$$Q=\{(m,b,a'):(m,b)\in P,a'\in A\}$$
$$\Pi=\{\pi:ax+y=m'z+b': (m',b')\in P, a\in A\},$$
and note that $|Q|=|\Pi|=|L||A|$, and also $E(L,A)=|I(Q,\Pi)|$.
We therefore by Theorem \ref{incidence} have 
\begin{eqnarray*}
E(L,A)&=&|I(Q,\Pi)|\\
&\leq&\frac{1}{q^{r}}|Q||\Pi|+q^{2r-1}|Q|^{\frac{1}{2}}|\Pi|^{\frac{1}{2}}\\
&=&\frac{1}{q^{r}}|L|^2|A|^2+q^{2r-1}|L||A|
\end{eqnarray*}
which proves the first part of the theorem.
If we denote 
$$r_{L(A)}(y)=|\{((m,b),a)\in P\times A : y=ma+b\}|,$$
then by the Cauchy-Schwarz inequality
\begin{eqnarray*}
|L|^2|A|^2&=&\left(\sum_{y}  r_{L(A)}(y) \right)^2\\
&\leq&|L(A)|\sum_{y}(r_{L(A)}(y))^2\\
&=&|L(A)|E(L,A)\\
&\ll&|L(A)|\left(\frac{1}{q^{r}}|L|^2|A|^2+q^{2r-1}|L||A|\right),
\end{eqnarray*}
therefore
$$|L(A)|\gg\min\left\{q^{r},\frac{|L||A|}{q^{2r-1}}\right\}.$$
\end{proof}

\begin{proof}[Proof of Theorem \ref{sumproduct}]
Note that $BA+C=L_{P}(A)$ where $P=B\times C$ so that $ L_P=L_{B\times C}$. Theorem \ref{collision} implies that
\begin{eqnarray}\label{abc}
|BA+C|=|L_{P}(A)|\gg\min\left\{q^{r},\frac{|A||B||C|}{q^{2r-1}}\right\}.
\end{eqnarray}
If we take $A=B=C$ in (\ref{abc}), we simply get
$$ |AA+A|\gg \min\left\{q^{r}, \frac{|A|^3}{q^{2r-1}}\right\}$$
and when $q^r<\frac{|A|^3}{q^{2r-1}}$, i.e.,  $|A|>q^{r-\frac{1}{3}}$, we have $|AA+A| \gg q^r$ which proves the claim.
\end{proof}

\section{Proof of Theorem \ref{second}}
For the proof of Theorem \ref{second} we will use the following Pl\"unnecke-Ruzsa inequality.
\begin{lemma}\label{Plunnecke}
 Let $A$ and $B$ be finite subsets of an abelian group such that $|A+B|\leq K|A|$. Then for an arbitrary $0<\delta<1$ there is a nonempty set $X\subset A$ such that $|X|\geq(1-\delta)|A|$ and for any integer k, one has 
\begin{eqnarray}
|X+kB|<\left(\frac{K}{\delta}\right)^k|X|.
\end{eqnarray}
\end{lemma}

\begin{proof}[Proof of Theorem \ref{second}]
Let $A^2+A^2=S$ and $|S|=K|A|$. we can use Lemma \ref{Plunnecke} to refine $A$ to a large subset $A'$ of $A$, and then for $k=2$ we will get
$$|{A'}^2+{A'}^2+{A'}^2|\ll K^{2}|A|.$$
Since working with $A'$ instead of $A$ will not change the calculations in the end, we replace $A$ with $A'$ using the same notations as $A$ and $S$.
Let
\begin{eqnarray*}
E&=&|\{c^2+a^2+{b'}^2={c'^2+{a'}^2+b^2: a,b,c,a',b',c'\in A}\}|\\
&=&|\{c^2+a^2-b^2={c'^2+{a'}^2-{b'}^2: a,b,c,a',b',c'\in A}\}|\\
&=&|\{c^2+2as-s^2={c'}^2+2a's'-{s'}^2: s=a+b, s'=a'+b'\}|
\end{eqnarray*}
Hence $E=E(L,A)$ where $L=L_P$
$$P=\{(2s,c^2-s^2): s\in A+A, c\in A\}.$$
We can see $E(L,A)$ as a point-plane incidence in $R^3$ and Theorem \ref{collision} with $|L|\leq |A+A||A|$ gives
\begin{eqnarray*}
E=E(L,A)&\leq& \frac{1}{q^r}|A+A|^2|A|^2|A|^2+q^{2r-1}|A+A||A||A|\\
&=& \frac{1}{q^r}|A+A|^2|A|^4+q^{2r-1}|A+A||A|^2
\end{eqnarray*}
From the Cauchy-Schwarz inequality, it follows that
\begin{eqnarray}\label{last}
|A|^{6}&\leq& |A^2+A^2+A^2|E\nonumber\\
&\ll&K^2|A|E\nonumber\\
&\leq&\frac{|A^{2}+A^{2}|^{2}}{|A|}\left\{\frac{1}{q^r}|A+A|^{2}|A|^{4}+q^{2r-1}|A+A||A|^2 \right\}
\end{eqnarray}
Note that when $|A+A||A|^{2}>q^{3r-1}$, the the first term in RHS of (\ref{last}) dominates and we obtain 
$$|A^2+A^2||A+A|\gg q^{\frac{r}{2}}|A|^{\frac{3}{2}}.$$

\end{proof}

\vskip.125in
\noindent \textbf{Acknowledgments.} The author  would like to thank Alex Iosevich and Jonathan Pakianathan for their valuable comments.

 \end{document}